\documentclass[10pt]{amsart}

\usepackage{latexsym}
\usepackage{amscd}
\usepackage{amssymb}
\usepackage[all]{xy}
\usepackage{amsmath}
\newtheorem{theorem}{Theorem}[section]
\newtheorem{proposition}[theorem]{Proposition}

\newtheorem{corollary}[theorem]{Corollary}
\newtheorem{remark}[theorem]{Remark}
\newtheorem{example}[theorem]{Example}
\newtheorem{conjecture}[theorem]{Conjecture}

\def\Soc{\mbox{Soc\/}}

\def\Im{\mbox{Im\/}}
\def\End{\mbox{End\/}}

\def\H{\mbox{Hom\/}}

\begin{document}

\title[pseudo-Frobenius rings and the FGF conjecture]{New Characterizations of pseudo-Frobenius rings and a generalization of the FGF conjecture}
\author{Pedro A. Guil Asensio}
\thanks{{The first author has been partially supported by the DGI (MTM2010-20940-C02-02)
and by the Excellence Research Groups Program of the S\'eneca Foundation of the Region of Murcia. Part of
the sources of both institutions come from the FEDER funds of the European Union.}}
\address{Departamento de Mathematicas, Universidad de Murcia, Murcia, 30100, Spain}
\email{paguil@um.es}
\author{Serap Sahinkaya}
\address{Gebze Institute of Technology, Department of Mathematics, Kocaeli, 41400 Turkey}
\email{ssahinkaya@gyte.edu.tr}
\author{Ashish K. Srivastava}
\address{Department of Mathematics and Computer Science, St. Louis University, St.
Louis, MO-63103, USA} \email{asrivas3@slu.edu}
\dedicatory{Dedicated to Alberto Facchini on his 60th birthday}

\begin{abstract}
We provide new characterizations of pseudo-Frobenius and quasi-Frobenius rings in terms of tight modules. In the process, we also provide fresh perspectives on FGF and CF conjectures. In particular, we propose new natural extensions of these conjectures which connect them with the classical theory of PF rings. Our techniques are mainly based on set-theoretic counting arguments initiated by Osofsky. Several corollaries and examples to illustrate their applications are given.
\end{abstract}

\maketitle



\section{Introduction.}

\noindent A ring $R$ is called right {\it pseudo-Frobenius} (PF, for short) when it is a right self-injective right cogenerator ring. And a right PF ring is called {\it quasi-Frobenius} (QF, for short) when it is, moreover, right (and left) artinian.  The origin of these rings can be drawn back to extensions of the concept of Frobenius algebras associated to the modular representations of finite groups (see e.g. \cite{CR}).

It is well known that a two-sided PF ring establishes a perfect duality in the sense
of \cite[Chapter 12, pages 307-308]{Ka} and that a left and right cogenerator ring (in particular,
a commutative cogenerator ring) is both-sided PF. The main reason why left and right cogenerator rings induce a perfect duality is that they are both-sided finitely cogenerated, that is, their left and right socles are finitely generated and essential in the ring (see e.g. \cite[Theorem 19.18]{Lam}).
One sided PF rings were introduced and studied independently by Azumaya \cite{Az}, Osofsky \cite{O} and Utumi \cite{U}. It is known that a right PF ring does not need to be left PF  \cite{DM}. But a deep theorem of Osofsky showed that right PF rings still enjoy the properties of being semiperfect and having finitely generated essential right socle \cite[Theorem~1]{O}.

On the other hand, it is easy to check that a ring in which any right module embeds in a free module is QF. This fact suggested Faith to conjecture in \cite{F1} that a ring is QF provided that any finitely generated right module embeds in a free module, thus extending an older question of Levy for commutative rings. And more generally, it is conjectured that a ring in which every cyclic right module embeds in a free module is right artinian. Rings satisfying that every cyclic (resp., finitely generated) right module embeds in a free module are usually called in the literature right CF (resp., right FGF) rings. And the question of whether any right CF (resp., right FGF) ring is right artinian (resp., QF) is nowadays known as the CF (resp., FGF) conjecture. Both conjectures are still open, whereas it is known that the CF conjecture implies the FGF conjecture and that they are true under many different additional hypothesis (see e.g. \cite{F1, EE, GG3, GT, JLP1, JST}).
Note that every right FGF ring is a right CF ring, but the converse  is not true. Bj\"{o}rk \cite{Bjork2} gave an example of a right CF ring which is not right FGF.

Probably the most promising partial positive results to the CF and FGF conjectures are based on using the set theoretical counting techniques developed by Osofsky in her proof that a right PF ring has finitely generated essential socle. This approach to the conjecture was initiated independently by  Bj\"{o}rk \cite{Bjork} and Tolskaya \cite{Tol} who proved that every right self-injective right CF ring is right artinian. And it culminated in \cite{EE}, where the authors proved that every ring in which any cyclic (resp., finitely generated) right module essentially embeds in a projective module is right artinian (resp., QF). They also proved in \cite{GG2} that a right CF and right extending ring has finitely generated essential socle. In particular, any right cogenerator right extending ring is right PF. Note that a ring is called right extending (or right CS) if every right ideal is essential in a direct summand of the ring.

All the above results suggest that there might exist a deep relation between the characterization obtained by Osofsky of right PF rings and the CF and FGF conjectures. But surprisingly, it seems that there has not been any attempt in the literature of connecting both situations. The main purpose of the present paper is to highlight these connections, which allows us to obtain new non-trivial characterizations of right PF rings, as well as new partial positive answers to the CF and FGF conjectures.

Our approach is based on the notion of tight rings. Tight rings and modules were introduced by Golan and L\'opez-Permouth in \cite{GLP} in order to study QI-filters and they have been later studied in \cite{JL, JLPS} in connection with weakly-injective modules. Recall that  a ring $R$ is called right tight (resp., right $R$-tight) if every finitely generated (resp., every cyclic) submodule of its injective envelope $E(R_R)$ embeds in $R$. The definition of tightness is closely related to the notion of embedding of finitely generated or cyclic modules in free modules. Therefore, it seems natural to conjecture that they might play a role in the characterization of right PF rings, as well as in answering the CF and FGF conjectures. Moreover, any right PF ring is trivially right tight and thus, they are the natural candidate to establish a link between both notions. And as a byproduct, one may adapt, exploit and extend different deep techniques, which have been developed in order to solve these conjectures, to get nontrivial new characterizations of PF and QF rings.

We begin by extending in Theorem~\ref{one} the techniques developed in \cite{EE}. This allows us to obtain as corollaries the main results of \cite{EE, GG2}. Next, we study in Theorem~\ref{nil} and Theorem~\ref{two} when a right tight cogenerator ring has a finitely generated essential right socle. Both results are inspired by the above mentioned transfinite counting arguments introduced by Osofsky in \cite{O} which, in turn, were based on an old result of Tarski on {\em almost disjoint} partitions of infinite sets \cite{Ta}. The obtained results allow us to establish the following conjecture:

\smallskip

\noindent{\bf Conjecture~1.} Every right cogenerator right $R$-tight ring is right PF.

\smallskip

\noindent We finish this section by proving this conjecture under different additional conditions and exhibiting several corollaries and examples which illustrate the applications and limits of the developed theory.

We begin Section 3 by observing that the obtained results naturally lead to establish the following new conjecture that encompasses the different open questions and conjectures existing on the topic:

\smallskip

\noindent{\bf Conjecture~2.} Every right Kasch generalized right ($R$-)tight ring has finitely generated and essential right socle.

\smallskip

\noindent Recall that a ring $R$ is called right Kasch when it cogenerates all simple right modules. In particular, any right cogenerator ring is right Kasch. And $R$ is called generalized right ($R$-)tight if every finitely generated (resp., cyclic) submodule of $E(R_R)$ embeds in a free module. It may be noted here that a positive solution to Conjecture 2 would imply affirmative answers to both Conjecture ~1 and the CF and FGF conjectures. Note also that Osofsky's characterization of right PF rings can be seen as a particular solution to this conjecture when the ring is assumed to be right self-injective.

We dedicate the rest of the paper to show that our new conjecture is satisfied when we assume the different additional conditions under which the CF and FGF conjectures are known to be true. This shows that this conjecture naturally extends the CF and FGF conjectures, and it connects them to Osofsky's work on PF rings. Moreover, as a byproduct of these results, we obtain new partial positive answers to the CF and FGF conjectures.

Throughout this paper, all rings $R$ will be associative and with identity,
and Mod-$R$ will denote the category of right $R$-modules. We will use the notation $M_R$ to stress the right $R$-module structure of a module $M$, when necessary. We will denote by $J(R)$, the Jacobson radical of a ring $R$ and by $Z(R_R)$, the singular right ideal of $R$ consisting of those elements of $R$ which have essential right annihilator.
We refer to \cite{AF, JST, Lam, NY, St} for all undefined notions used in the text.

\bigskip

\section{New characterizations of PF rings.}


\noindent We begin by proving several extensions of \cite[Theorem~1]{O} which will be used in our characterization of right PF rings. As a consequence, we will also deduce the main results of \cite{EE,GG2}. Recall that a ring $R$ is called right Kasch if every simple right module embeds in $R$.


\begin{theorem}\label{one}
Let $R$ be a right Kasch ring such that each cyclic submodule of the injective envelope $E(R_R)$ embeds in a free module. Assume that every direct summand of $E(R_R)$ contains an essential projective module $P$ such that $P/(P\cdot Z(R_R))$ is finitely generated. Then $R_R$ has a finitely generated essential socle.
\end{theorem}

\begin{proof}
Let $E=E(R_R)$. As in \cite[Lemma 2.4]{EE}, we first show that if $S = \End(E_R)$
and $\{C_k\}_{k\in K}$ is an idempotent-orthogonal family of simple right
$S/J$-modules (with $J = J(S)$), then there exists an injective mapping
from index set $K$ to the set $\Omega(R)$ of isomorphism classes of simple right $R$-modules.

Since idempotents of $S/J$ lift modulo $J$, there exist idempotents $\{e_k\}_{k\in K}$ of $S$ such that $C_ke_k\neq 0$ for any $k\in K$ and either $C_je_k=0$ or $C_ke_j=0$ for $k\neq j$. Let $c_k\in C_k$ be such that $c_ke_k\neq 0$ for each $k\in K$, and let $p_k: S_S\rightarrow C_k$ be the homomorphism defined as $p_k(1)=c_ke_k$. If $e_{k^{*}}=\H_R(E, e_k)$ is the endomorphism of $S_S$ given by left multiplication with $e_k$, we have $(p_k \circ e_{k^{*}})(1)=c_ke_k^2=c_ke_k=p_k(1)$, and so $p_k \circ e_{k^{*}}=p_k$. Thus it follows that
\[(p_k \otimes_S E) \circ e_k=(p_k \otimes_S E) \circ (e_{k^{*}}\otimes_S E)=(p_k \circ e_{k^{*}}) \otimes_S E=p_k\otimes_S E.\]

\noindent Set $E_k = \Im(e_k)$. We have by hypothesis that $E_k = E(P_k)$ with $P_k$, a projective module such that $P_k/(P_k\cdot Z(R_R))$ is finitely generated. By hypothesis, each finitely generated submodule
of $E_k$ embeds in a free module. Then $(p_k\otimes_S E)(P_k)\neq 0$ by \cite[Proposition~1.3]{CPI}.  Let $h_k: P_k\rightarrow E_k$, $i_k: E_k\rightarrow E$,
and $t_k=i_k \circ h_k: P_k\rightarrow E$ be the inclusions, and set $L_k:=Im ((p_k\otimes_S E)\circ t_k)$, with canonical projection $q_k: P_k\rightarrow L_k$ and
inclusion $w_k:L_k\rightarrow C_k\otimes_S E$. Note that $C_k$ is a right $S/J$-module and thus, $L_k$ is a right $R/Z(R_R)$-module. Therefore, $(p_k\otimes_S E)\circ t_k$ factors through $P_k/(P_k\cdot Z(R_R))\cong P_k\otimes_R(R/Z(R_R))$ and so, $L_k$ is finitely generated. This means that  we can choose for each $k\in K$, a simple quotient $U_k$ of $L_k$ with canonical projection $\pi_k: L_k\rightarrow U_k$. We define a map from index set $K$ to the set $\Omega(R)$ by assigning $k\mapsto [U_k]$, where $[U_k]$ denotes the isomorphism class of the simple module $U_k$. It may be checked that this map is injective.

Now, since $R_R$ cogenerates the simple modules by hypothesis, we have, as shown in \cite{EE}, that $\left| \Omega(R) \right| \leq \left| C(R)\right|$ where $C(R)$ denotes a set of representatives of the isomorphism classes of simple submodules of $R$. Let $\mathcal M$ represent the set of isomorphism classes of minimal right ideals of $S/J$ and assume $|\Omega (R)| =n$. We claim that $|{\mathcal M}| = n$. Let $C_1, \ldots,
C_r$ be a set of representatives of the elements of $\mathcal M$. Suppose that there exists a simple right $S$-module $C=C_{r+1}$ which is not isomorphic to any of the $C_i,$ for $1\le i\le r$. There exist idempotent elements
$e_1, e_2, \ldots, e_r \in S$ such that, if $\bar e_i = e_i+J$, then $C_i = \bar e_i(S/J)$ for each $1\le i\le r$. Since $\bar e_i(S/J)
\bar e_j = \H_{S/J}(\bar e_j(S/J),\bar e_i(S/J))$, we have $\bar e_i(S/J)\bar e_j = 0$ for $i, j\le r$, $i\neq j$ and $\bar e_i(S/J)\bar e_i \neq 0$ for all $i = 1, \ldots, r$. Thus the family $\{C_i\}$, $i=1,\ldots,r+1$
is an idempotent-orthogonal family of simple right $S/J$-modules with
respect to the idempotents $\{\bar e_1, \ldots, \bar e_r, 1\}$. We have then $r+1 \leq n$, a contradiction that shows that the simple module $C$ cannot exist, and hence that $S/J$ is a semisimple artinian ring. Therefore $S$ is
a semiperfect ring and $E_R$ is a finite-dimensional module. Thus
$E_R$ is a finite direct sum of indecomposable submodules. From the preceding argument it also follows that $r\le n$ and hence that $r=n$. Since there exists a bijection between the set of
isomorphism classes of indecomposable direct summands of $E_R$ and the
set $\mathcal M$ of isomorphism classes of minimal right ideals of $S/J$, the number of isomorphism classes of indecomposable direct summands of $E_R$ is exactly $n$, and so each of them is an injective envelope of a simple right $R$-module, so that
$E_R$, and hence $R_R$ has finite essential socle.
\end{proof}


\noindent As a consequence, we have the following.

\begin{corollary}\cite[Corollary~3.3, Corollary~3.5]{EE}\label{EE} Let $R$ be a ring. If every cyclic right $R$-module essentially embeds in a projective module, then $R$ is right artinian.

If moreover, every finitely generated right $R$-module essentially embeds in a free module, then $R$ is QF.
\end{corollary}

\begin{proof}
Let us first show that $\Soc(R_R)$ is finitely generated and essential. In order to apply Theorem~\ref{one}, we only need to show that any direct summand of $E=E(R_R)$ contains an essential projective submodule $P$ such that $P/(P\cdot Z(R_R))$ is a finitely generated right module.

Let $E'$ be a nonzero direct summand of $E$. As $R_R$ is essential in $E$, $E'$ contains an essential cyclic module $xR$. By hypothesis, there exists an essential monomorphism $u:xR\rightarrow P$, for some projective module $P$. And this essential monomorphism extends by injectivity to a monomorphism $v:P\rightarrow E'$. Therefore, $E'$ contains the essential projective submodule $P$.

Let us now check that $P/(P\cdot Z(R_R))$ is finitely generated.
As $P$ is projective, it is a direct summand of a free module, say $R^{(I)}$.
Let $w:P\rightarrow R^{(I)}$ and $p:R^{(I)}\rightarrow P$ be the canonical injection and projection. Now, as $xR$ is cyclic, there exists a finite subset $I'$ of $I$ such that $w\circ u|_{xR}\subseteq R^{(I')}$. Let $\pi:R^{(I)}\rightarrow R^{(I')}$ and $i:R^{(I')}\rightarrow R^{(I)}$ be the projection and injection, respectively. Then $p\circ w|_{xR}-p\circ i\circ \pi\circ w|_{xR}=0$ and therefore, as $xR$ is essential in $E'$, this means that $\Im(p\circ w-p\circ i\circ \pi\circ w)\in P\cdot Z(R_R)$. Therefore, as $p\circ w=1_P$, we deduce that $$P/(P\cdot Z(R_R))=(\Im(\pi\circ i\circ\pi\circ w)+P\cdot Z(R_R))/(P\cdot Z(R_R))$$ and therefore, it is finitely generated as it is a homomorphic image of $R^{(I')}$. The proof now follows from the arguments used in \cite[Corollary~3.3, Corollary~3.5]{EE}.
\end{proof}


\noindent Recall that a ring $R$ is called right extending (or right CS) if every right ideal essentially embeds in a direct summand of $R$.


\begin{corollary}\cite[Corollary~2.7]{GG2} Let $R$ be a right Kasch ring. If $R_R$ is extending, then it has finitely generated essential socle.
\end{corollary}

\begin{proof} If $R_R$ is extending, then clearly any direct summand of $E(R_R)$ contains an essential direct summand of $R$. So the result follows from Theorem~\ref{one}.
\end{proof}

\noindent Thus if $R$ is right extending and each cyclic right $R$-module embeds in a free module, then each cyclic right $R$-module has finitely generated essential socle and consequently the ring is right artinian.

\begin{corollary} \cite[Corollary ~2.9]{GG2} If $R$ is a right extending ring, then both CF and FGF conjectures hold for $R$.
\end{corollary}


\noindent The next theorem will be essential for obtaining our new characterizations of right PF rings. Its proof is based on transfinite counting arguments inspired by  \cite[Theorem~1]{O} and \cite[Theorem 6]{GT}. We will say that a ring $R$ has {\em completely nil Jacobson radical} if for any two-sided ideal $N$ of $R$, any element in the Jacobson radical of $R/N$ is nilpotent.


\begin{theorem}\label{nil}
Let $R$ be a right cogenerator right $R$-tight ring. If $R/Z(R_R)$ has completely  nil Jacobson radical, then $\Soc(R_R)$ is finitely generated and essential in $R_R$.
\end{theorem}

\begin{proof}
Let us first fix our notation. We will denote the injective envelope of $R_R$ by  $E=E(R_R)$ with inclusion $u:R_R\rightarrow E$. Let us set $S=\End_R(E)$ and $J=J(S)$, where $J(S)$ is  the Jacobson radical of $S$. It is shown in \cite[Lemma~1]{GT} that there exists a homomorphism of rings $\Phi:R\rightarrow S/J$ which assigns  any element $r\in R$ to the element $s_r+J$, where $s_r$ is an endomorphism of $E$ which extends the left multiplication by $r$. The kernel of $\Phi$ is the singular ideal $Z(R_R)$ of $R_R$. Therefore, we get an injective homomorphism of rings $\Psi: R/Z(R_R)\rightarrow S/J$ induced by $\Phi$.

Let us now show that $\Soc(R_R)$ is finitely generated and essential in $R_R$.
We are going to prove it in three steps, as in \cite[Theorem 6]{GT}.


{\bf Step 1.} We claim that $\Soc(R_R)$ contains only finitely many homogeneous components.


\noindent We will assume that ${\rm Soc}(R_R)$ has infinitely many homogeneous components and we will try to reach a contradiction.  Let  $\{C_i\}_{i\in I}$ be a representative set of the isomorphism classes of simple modules in ${\rm Soc}(R_R)$. By Tarski's Lemma~\cite{Ta} (see also \cite{O}), there exists a  a family $\mathcal K$ of {\em almost disjoint subsets} of $I$ such that $|\mathcal{K}|\gneqq |I|$ and $I$ is the union of the sets in this family. In other words there exists a family $\mathcal K$ of subsets of $I$ satisfying that:
\begin{itemize}
\item $I=\cup_{K\in\mathcal{K}}K$ and  $|\mathcal{K}|\gneqq |I|$.
\item $|K|$ is infinite for any $K\in\mathcal K$.
\item $|K|=|K'|\gneqq |K\cap K'|$ for any $K,K'\in\mathcal{K}$ with $K\neq K'$.
\end{itemize}

We know that there exists an injective map from the index set I to the family of isomorphism classes of minimal right ideals of $S/J$ \cite[Lemma 3]{GT}. This map assigns any element $i\in I$ to the minimal right ideal $e_iS/e_iJ$ of $S/J$, where $e_i\in S$ is an idempotent such that $e_iE=E(C_i)$.

 Let us take any subset of $I$, say $A$, and set
$$X_A=E\left(\sum \{D\leq\,S/J_
{S/J} \,|
\, D\cong e_{C_i}S/e_{C_i}J \text{ for some } i\in A\}\right)$$

Since $X_A$ is a direct summand of $E$ , there exists an idempotent $e_A\in S$ such that $X_A=e_AE$. We know by \cite[Lemma 4]{GT} that $e_A+J$ is a central idempotent in $S/J$. In particular, $((1-e_I)S+J)/J$ is a two sided ideal of $S/J$ when $A=I$. For simplicity,  $((1-e_I)S+J)/J$ will be denoted by $\mathcal{N}_I/J$. Since $\mathcal{N}_I/J$ is a two-sided ideal, its inverse image $\Psi^{-1}(\mathcal{N_I}/J)$ is a two sided ideal of $R/Z(R_R)$ and we will denote it by $\mathcal{M}_I/Z(R_R)$.

Let $\aleph=|K|$ and let us set
$$\mathcal{N}/J=\mathcal{N}_I/J + \sum\{(e_AS+J)/J\,|\, A\subseteq I\text{ with }|A|\lvertneqq \aleph\}$$
and call $\mathcal{M}/Z(R_R)=\Psi^{-1}(\mathcal{N}/J)$. By \cite[Lemma 7]{GT}, we know that $\{e_K+\mathcal{N}\,|\,K\in\mathcal{K}\}$ is an orthogonal family of nonzero central idempotents in $S/\mathcal{N}$.

Let $u$ be the inclusion of $R_R$ in its injective envelope $E$ and call $x_K=e_{K}\circ u(1)\in E$. Then $e_{K}\circ u$ factors as

$$
\begin{xy}
  \xymatrix{
    R \ar[d]_u\ar[r]^{f_{K}} & x_{K}R\ar[d]^{u_{K}} \\
    E \ar[r]_{e_K} & E  \\
    }
\end{xy}
$$
where $f_{K}$ is an epimorphism and $u_{K}$, a monomorphism. There exists a monomorphism $\alpha_{K}:x_{K}R\rightarrow R$ by our assumption that any cyclic  submodule of $E(R_R)$ embeds in $R$. By injectivity,  $\alpha_{K}$ extends to  an $s_{K}:E\rightarrow E$ such that $u\circ\alpha_{K}=s_{K}\circ u_{K}$.
Again, as  $s_{K}|_{E(x_{K}R)}:E(x_{K}R)\rightarrow E$ is a monomorphism, there exists an $h_{K}:E\rightarrow E$ such that $h_{K}\circ s_{K}\circ e_{K}=e_{K}$.

$$\begin{xy}
  \xymatrix{
   R \ar[d]_u\ar[r]^{f_{K}} & x_{K}R\ar[d]^{u_{K}}\ar [r]^{\alpha _{K}} &
   R\ar[d]^{u} \\
    E \ar[r]_{e_K} & E\ar[r]_{s_{K}}& E  \ar[r]_{h_{K}} &E.\\
    }
\end{xy}
$$

Call $r_K=s_K\circ e_K\circ u(1)\in R$. Our claim is that $r_K+\mathcal{M}\notin J(R/\mathcal{M})$. Assume otherwise that $r_{K}+\mathcal{M}\in J(R/\mathcal{M})$. As we are assuming that   $J(R/Z(R_R))$ is completely nil, we deduce that any element in $J(R/\mathcal{M})$ is nilpotent.  Thus, there exists a natural number $m\geq 1$ such that $r_{K}^m+\mathcal{M}=0$ in $R/\mathcal{M}$. But then, $s^m_K\circ e_K+\mathcal{N}=\Phi(r^m_K+\mathcal{M})=0$. Therefore, we get that

$$\begin{array}{ll}
0 & =h_K^m\circ s_K^m\circ e_K + \mathcal{N} = h_K^{m-1}\circ (h_K\circ s_K\circ e_K) \circ s_K^{m-1} +\mathcal{N}
\\ & = h_K^{m-1}\circ e_K \circ s_K^{m-1}+\mathcal{N}=\ldots = h_K\circ s_K\circ e_K+\mathcal{N}=e_K+\mathcal{N}
\end{array}.$$
But it is a contradiction since  $e_K$ does not belong $\mathcal{N}$ by construction.

 As $r_K+\mathcal{M}\notin J(R/\mathcal{M})$, there exists a maximal right ideal $L_K/\mathcal{M}$ of $R/\mathcal{M}$ such that $r_K+\mathcal{M}\notin L_K/\mathcal{M}$.  Thus, $R/L_K$ is a simple right $R$-module which satisfies  $R/L_K\cdot (r_K+\mathcal{M})\neq 0$.

We finally claim that $R/L_K\ncong R/L_{K'}$  when $K\neq K'$ with $K,K'\in \mathcal{K}$. Assume that   $\delta:R/L_K\rightarrow R/L_{K'}$ is an isomorphism. We get that
$$0\neq \delta(r_K+L_K)=\delta(1+L_K)\cdot (r_K+\mathcal{M}).$$

In particular, $\delta(1+L_K)\neq 0$. On the other hand, $\delta(1+L_K)$ is a generator of $R/L_{K'}$, since it is simple. And we know that $(R/L_K')\cdot (r_{K'}+\mathcal{M})\neq 0$, which assures the existence of an $r\in R$ such that $0\neq \delta(1+L_K)\cdot (r_Krr_{K'}+\mathcal{M})$. But then, $r_Krr_{K'}+\mathcal{M}\neq 0$ and thus, $\Psi((r_Krr_{K'}+\mathcal{M}))\neq 0$, by the injectivity of  $\Psi$ . We deduce that
$$s_K\circ h_K\circ e_K\circ \Psi(r+\mathcal{M})\circ s_{K'}\circ h_{K'}\circ e_{K'}+\mathcal{N}\neq 0.$$

But both idempotents are central in $S/\mathcal{N}$ and so,  $e_K\circ e_{K'}\notin \mathcal{N}$. And this means that $K=K'$, since otherwise  $e_K\circ e_{K'}\in \mathcal{N}$ by construction.

We have constructed then a family $\{R/L_K\}_{K\in\mathcal K}$ of non isomorphic simple right $R$-modules with  $|\mathcal{K}|\gvertneqq |I|$ isomorphism classes of right simple modules. This is a contradiction since we have assumed that $R_R$ has $|I|$ non isomorphic classes of simple right modules. Therefore, $Soc(R_R)$ must have  only finitely many homogeneous components.


{\bf Step 2.} We claim that any homogeneous component of ${\rm Soc}(R_R)$ is finitely generated.


\noindent We know by Step 1 that there are only finitely many homogeneous components in ${\rm Soc}(R_R)$. Let $\{C_1,\ldots,C_m\}$ be a representative set of simple modules belonging to them.
As $R_R$ is a cogenerator, there exist sets of orthogonal idempotents $\{r_i\in R\,|\,i=1,\ldots,m\}$ and $\{e_i\in E\,|\,i=1,\ldots,m\}$ such that $E(C_i)=r_iR=e_iE$ for each $i=1,\ldots,m$. In particular, this means that $\Psi(r_i+Z(R_R))=e_i+J$. And $r_i+Z(R_R)$ does not belong to the Jacobson radical of $R/Z(R_R)$ for any $i=1,\ldots,m$, since they are idempotent. Moreover,
 $E(C_i)$  is the projective cover of a simple right module $D_i$ since they are indecomposable injective direct summands of $R_R$.
Note also that,  $D_i\ncong D_j$ when $i\neq j$ and $D_i\cdot r_i\neq 0$ for each $i\in I$, by construction.

Now assume that some homogeneous component is not finitely generated. Say that it is the homogeneous component associated to $C_1$.
Repeating  the arguments in Step 1, but replacing $K$ by $A=\{i\}$ and $\mathcal{N}$ by $\mathcal{N}'=\Soc(S/J_{S/J})$, we may construct a central idempotent $e_A+J\in S/J$ such that $(e_AS+J)/J$ is the injective envelope of the homogeneous component corresponding to $e_1S/e_1J$ inside $S/J$. And we can find an element $r_A+\Psi^{-1}(\mathcal{N}')$ which does not belong to $J(R/\Psi^{-1}(\mathcal{N}'))$ and  $\Psi(r_A+J)=h_A\circ s_A\circ \circ e_A\notin \Soc(S/J_{S/J})$.
Let us choose a maximal right ideal $L/\Phi^{-1}(\mathcal{N}')$ of $R/\Phi^{-1}(\mathcal{N}')$ satisfying that $r_A+ L\neq 0$ in $R/\Phi^{-1}(\mathcal{N}')$. Note that this maximal right ideal does exist since
$r_A+\Phi^{-1}(\mathcal{N}')\notin J(R/\Phi^{-1}(\mathcal{N}'))$.
This means that, if we set $D=R/L$, this is a simple right $R$-module such that $D\cdot (r_A+R/\Phi^{-1}(\mathcal{N}'))\neq 0$.

We claim that $D$ is not isomorphic to $D_i$ for any $i=1,\ldots,m$. Assume on the contrary  that $\delta:D\rightarrow D_i$ is an isomorphism for some $i=1,\ldots,m$.
And fix a nonzero element $x\in D_i$ such that $x\cdot(r_i+Z(R_R))\neq 0$. This means that $\delta(x)\cdot(r_i+Z(R_R))\neq 0$ and so, it is a generator of $R/L$. Since $D\cdot(r_A+\Phi^{-1}(\mathcal{N}'))\neq 0$,  there exists an $r\in R$ such that $x\cdot(r_i r r_A+J)\neq 0$
in $D$.  And therefore,  $r_Arr_i\notin \Phi^{-1}(\mathcal{N}')$, because $D$ is a right $R/\Phi^{-1}(\mathcal{N}')$-module.
But then, $e_i\circ \Phi(r+J)\circ h_A\circ s_A\circ e_A+J=\Phi(r_irr_A+ Z(R_R))\notin \Soc(S/J_{S/J})$.
And this is a contradiction since $(e_i+J)/J\in {Soc}(S/J)$ and ${\rm Soc}(S/J)$ is a two sided ideal.

We have shown that each homogeneous component of ${Soc}(R_R)$ is finitely generated and so we proved that  ${Soc}(R_R)$ is finitely generated.


{\bf Step 3.} We finally claim that $\Soc (R_R)$ is essential in $R_R$.


\noindent Repeating the arguments of Step 2, we may construct sets of orthogonal idempotents $\{r_1,\ldots,r_m\}$ and $\{e_1,\ldots,e_m\}$ in $R$ and $S$ associated to a representative family $\{C_1, \dots, C_m \}$ of the isomorphism classes of the simple right ideals of $R$  such that $E(C_i)=e_iE=r_iR$ and  $\Psi(r_i+Z(R_R))= e_i+J$ for each $i=1,\ldots,m$.
Let  $D_i$ be  a simple module such that $E(C_i)$ is a projective cover of $D_i$. We get that $D_i\cong D_j$ when $i\neq j$ and $D_i\cdot r_i\neq 0$ for each $i=1,\ldots, m$.

Assume that $\Soc(R_R)$ is not essential in $R_R$. This means that $E(\Soc (R_R))\neq
E(R_R)$. Let $e_I\in S$ be the idempotent such that  $E(\Soc (R_R))=e_IE$. The arguments of Step ~1 show that $e_I+J$ is central in $S/J$.

Repeating the arguments used in Step 2, but by replacing the idempotent $e_A$ by  $1-e_I$ and the ideal $\mathcal{N}'$ by $J$, we get an $r_I\in R$ such that $r_I+Z(R_R)\notin J(R/Z(R_R))$ and elements $s_I,h_I\in S$ such that $h_I\circ s_I\circ (1-e_I)\notin J$ and $\Phi(r_I+Z(R_R))=s_I\circ (1-e_I)+J$. Therefore, there exists a maximal right ideal $L/Z(R_R)$ of $R/Z(R_R)$ such that $r_I+Z(R_R)\notin L/Z(R_R)$ . In particular, if we call $D=R/L$, we get that $D\cdot (r_I+Z(R_R))\neq 0$.

We claim that $D\ncong D_i$ for any $i=1,\ldots,m$. Let us assume on the contrary that $\delta:D\rightarrow D_i$ is an isomorphism and choose an $x\in D$ such that $x\cdot(r_I+Z(R_R))\neq0$. We then obtain that $\delta(x)\cdot (r_I+Z(R_R))\neq 0$ as in Step 2. And thus, there exists an $r\in R$ satisfying that $\delta(x)\cdot(r_Irr_i+Z(R_R))\neq 0$.
Hence $r_Irr_i+Z(R_R)\neq 0$ and therefore $h_I\circ s_I\circ (1-e_I)\circ \Psi(r+Z(R_R))\circ e_i+J\neq 0$ in $S/J$.
And, as $(1-e_I)+J$ is central in $S/J$, we deduce that $(1-e_I)\circ e_i+J\neq 0$. But this is not possible because $e_I\cdot e_i=e_i$ by construction of $e_I$. So we get a contradiction which shows $\Soc(R_R)$ is essential in $R_R$.
\end{proof}


\noindent We can now state our new characterizations of right PF rings.

\begin{theorem}\label{two} Let $R$ be a ring. Then the following conditions
are equivalent:
\begin{enumerate}
	\item $R$ is right PF.
	\item $R_R$ is a cogenerator and every cyclic submodule of $E(R_R)$ essentially embeds in a projective module.
	\item $R$ is a right ($R$-) tight cogenerator and $J(R/Z(R_R))$ is completely nil.
	\item $R_R$ is tight, $R$ is semilocal, and the injective envelopes of simple right $R$-modules are finitely
	generated.
\end{enumerate}
\end{theorem}
\begin{proof} (1) $\Rightarrow$ (2). This is straightforward.

(1) $\Rightarrow$ (3). $R$ is a right self-injective semiperfect ring. Therefore, $J(R)=Z(R_R)$ and $R/Z(J(R))$ is von Neumann regular. Thus, any ring which is a homomorphic image of $R/J(R)$ has zero Jacobson radical.

(2) or (3) $\Rightarrow$ (4). By Theorem~\ref{one} and Theorem~\ref{nil}, we get that $\Soc(R_R)$ is finitely generated and essential in $R$. In particular, there exists a finite number of isomorphism classes of simple modules, say $\{C_1,\ldots,C_m\}$.  Moreover, as $R_R$ is a cogenerator, the injective envelopes of simple right $R$-modules are direct
summands of $R$ and thus, they are projective and finitely generated.
Since each
$E(C_i)$ is projective, it is a local module and hence it is the projective cover of the simple module $E(C_i)/E(C_i)J(R)$. Note that $E(C_i)/E(C_i)J(R)$ is not isomorphic to $E(C_{i'})/E(C_{i'})J(R)$ if $i\neq i'$. Thus, each simple right module has a projective cover and this means that $R$ is semiperfect by \cite[Theorem~27.6]{AF}. Therefore, $R$ is semilocal.

(4) $\Rightarrow$ (1). We first show that $\Soc(R_R)$ is finitely generated. As we know that $R$ is semilocal, we may write $R/J=\oplus_{i=1}^nD_i$ with each $D_i$, a simple module. Assume that $\Soc(R_R)$ is not finitely generated and choose a direct sum $\oplus_{k=1}^{n+1}C_k$ of simple modules in $\Soc(R_R)$. By hypothesis, $\oplus_{k=1}^{n+1}E(C_k)$ is a finitely generated submodule of $E(R_R)$ and thus, it embeds in $R$. This means that there exists a set $\{e_k\,|\,k=1,\ldots,n+1\}$ of nonzero orthogonal idempotents in $R$ such that $E(C_k)=e_kR$ for each $k=1,\ldots,n+1$. But this means that $\oplus_{k=1}^{n+1}e_kR/e_kJ$ is a direct sum of $n+1$ nonzero submodules of $R/J$. A contradiction, since $R/J$ is a semisimple ring of length $n$.

Now we claim that $\Soc(R_R)$ is essential in $R$. We know that $\Soc(R_R)=\oplus_{k=1}^nC_k$ is finitely generated. So $\oplus_{k=1}^nE(C_k)$ is also finitely generated by hypothesis and it embeds in $R$. Let $e_k\in R$ be an idempotent such that $E(C_k)=e_kR$ for each $k$. Assume that $\oplus_{k=1}^nE(C_k)$ is not essential in $R$ and call $e=1-\sum_{k=1}^n e_k$. Then $R/J=(\oplus_{k=1}^ne_kR/e_kJ)\oplus eR/eJ$. Again a contradiction since the length of $R/J$ is $n$.

Therefore, $R=\oplus_{k=1}^n e_kR\cong \oplus_{k=1}^nE(C_k)$ is a right self-injective ring and thus $R/J=\oplus_{k=1}^n e_kR/e_kR$. Let us note that if
$e_kR/e_kJ\ncong e_{k'}R/e_{k'}J$, then $e_kR\ncong e_{k'}R$ either, as they are their projective covers. And thus, $C_k\ncong C_{k'}$. This means that $R$ must contain all isomorphism classes of simple right $R$-modules and so, it is a right cogenerator.
\end{proof}
	

\noindent Motivated by the above theorem, we would like to propose the following conjecture.

\begin{conjecture} \label{conj2}
If $R$ is a right cogenerator right $R$-tight ring, then $R$ is right PF.
\end{conjecture}

We are now going to obtain several corollaries of Theorem \ref{two} which will give partial answers to the above proposed conjecture. Recall that a ring $R$ is called right {\it automorphism-invariant} if it is invariant under any automorphism of its injective envelope $E(R_R)$ (see e.g. \cite{GSr}).

\begin{corollary}\label{ai}
Let $R$ be a right $R$-tight, right automorphism-invariant ring such that $R_R$ is a cogenerator. Then $R$ is right PF.
\end{corollary}

\begin{proof}
It is shown in \cite{GSr} that $R$ is semiregular and $Z(R_R)$ is the Jacobson radical of $R$. Therefore, $R/Z(R_R)$ is von Neumann regular and this means that the Jacobson radical of $R/Z(R_R)$ is completely nil.
The result now follows from Theorem~\ref{two}
\end{proof}


It is clear that if $R$ is a right extending ring, then
every cyclic submodule of $E(R_R)$ essentially embeds in a projective module if and only if $R_R$ is $R$-tight. Therefore, we have:

\begin{corollary}\label{three} Let $R$ be a ring. Then the following
conditions are equivalent:
\begin{enumerate}
	\item $R$ is right PF.
	\item $R_R$ is $R$-tight, extending and a cogenerator.
	\item $R_R$ is $R$-tight, extending and the injective envelopes of simple right  $R$-modules are projective.
\end{enumerate}
\end{corollary}
\begin{proof} The implications $(1)\Rightarrow (2)\Rightarrow (3)$ are clear.

$(3)\Rightarrow (1)$ Clearly $R_R$ is a cogenerator. Using the results of \cite{EE} it is possible
to show that every cyclic submodule of $E(R_R)$ essentially embeds in a projective module, for if $R_R$ is $R$-tight and extending,
then every cyclic submodule of $E(R_R)$ is essentially embeddable in a
projective module and so we may apply \cite[Theorem 3.1]{EE}. The implication now follows from Theorem~\ref{two}.
\end{proof}


Our first  example shows that we cannot drop from Theorem~\ref{two} and Corollary~\ref{three} the hypothesis that $R_R$ is a cogenerator.

\begin{example}\label{integers}
The ring of rational integers {\bf Z} is
both tight and extending but it is not self-injective. Therefore, it is not PF.
\end{example}

The next example shows that a right $R$-tight ring does not need to be right extending.
\begin{example}\label{triangular}
The ring $R$ of upper triangular matrices over a field
$F$ is right $R$-tight. Furthermore, since
$E(R_R)$ is projective, every direct summand  of $E(R_R)$ has an
essential finitely generated projective submodule. However, $R$ is
not a right extending ring, for the right ideal
$\{\begin{pmatrix}0~x\cr 0~x\end{pmatrix}|x\in F\}$ is not essential in a direct summand
of $R_R$.
\end{example}

Note that the ring  constructed in the above example is not right  tight  since otherwise it would be quasi-Frobenius as it is right artinian. Next, we give example of a ring $R$ such that every finitely generated submodule of its injective envelope $E(R_R)$ embeds in a free module but $R$ is not  right extending.

\begin{example}\label{extending}
Let $R$ be a right noetherian ring such that the injective envelope of any flat module is flat. For instance, a commutative noetherian domain (see \cite[Theorem 3]{CE}). As this property is clearly Morita invariant, any flat right module over $M_n(R)$ has a flat injective envelope for any $n\geq 1$. Let $E=E(M_n(R)_{M_n(R)})$ and let $p:M_n(R)^{(I)}\rightarrow E$ be an epimorphism. As $E$ is flat, $p$ is a pure epimorphism. Let $N$ be any finitely generated submodule of $E$. As $R$ is right noetherian, $N$ is finitely presented and thus, the inclusion $i:N\rightarrow E$ lifts to a monomorphism $v:N\rightarrow M_n(R)^{(I)}$. Therefore, $M_n(R)$ is a ring such that every finitely generated submodule of its right injective envelope embeds in a free module. However, if $R$ is a commutative noetherian domain which is not semihereditary, then $M_n(R)$ is not right (nor left) extending (see \cite[Example 2.3.13]{BPR}) .
\end{example}


Finally, we exhibit an example of a commutative ring $R$ which is tight, but  it does not have the property that every direct summand of $E(R_R)$ has
an essential finitely generated projective  submodule, nor every cyclic
submodule of $E(R_R)$ is essentially embeddable in a projective. In particular, $R$ is not extending.


\begin{example}
Let $R = \{(m,n)\in {\bf Z}\times {\bf Z}| m \equiv n (mod ~2)\}
\subseteq {\bf Z}\times {\bf Z}$. Then $R$ is a semiprime Goldie ring and,
in fact, $R$ is an order in the semisimple ring ${\bf Q}\times {\bf Q}$,
so that $E(R_R) \cong {\bf Q}\times {\bf Q}$. Using
\cite[Proposition 4.2]{JLP1} one can easily see that
as R is a semiprime (two-sided) Goldie ring, then $R_R$ is tight, cf. \cite{JL}.

On the other hand, the  principal ideal $K_R = R(2,0) = 2{\bf Z}\times 0
\subseteq R$  is an essential submodule of ${\bf Q}\times 0 = \{(q,0)|q\in
{\bf Q}\}$ and, since the latter module is divisible, it is injective and
hence $E(K_R) \cong {\bf Q}\times 0$. Assume then that  $K_R$ is essentially
embeddable
in a projective module $P_R$. Then $P_R$ embeds in ${\bf Q}\times 0$.
On the other hand, if we set $L = R(0,2) \subseteq R$, then it is easily
checked that $R/L\cong {\bf Z}$ and that $P = P/LP$ is also a projective
$R/L$-module. Thus $P$ can be viewed as a ${\bf Z}$-projective submodule
of ${\bf Q}\times 0$ and this implies that $P$ is cyclic as $R/L$-module
and hence as $R$-module. Thus there exists $0 \neq  (q,0) \in
{\bf Q}\times 0$ such that $P = R(q,0)$. The map  $K \rightarrow P$ defined
by $(2n,0) \mapsto (qn,0)$ is easily seen to be an isomorphism, and so we
must have $P_R \cong K_R$. But, since $R_R$ is indecomposable, it is
clear that $K_R$ is not projective, which gives a contradiction and shows
that $K_R$ is not  essentially embeddable in a projective module. Observe
also that, in particular, $R$ cannot be a extending ring.
\end{example}

It is well known that a  ring $R$ is right PF if and only if it is
right self-injective and has finite essential socle (i.e., $R_R$  is finitely cogenerated). The following result extends this fact.

\begin{proposition}\label{four} Let $R$ be a ring such that $R_R$ is tight
and $E(R_R)$ is both finitely generated and finitely cogenerated. Then $R$
is right PF.
\end{proposition}
\begin{proof} Since $R_R$ is tight and $E(R_R)$ is finitely generated, $E(R_R)$
embeds in $R_R$ and so there exists $X \subseteq R_R$ such that
$R_R \cong E(R_R) \oplus X$. Then $Soc(R_R) \cong Soc(E(R_R)) \oplus
Soc(X) \cong Soc(R_R) \oplus Soc(X)$ (since $Soc(R_R)$ is essential in
$R_R$). Now, since $Soc(R_R)$ is finitely generated, we see by Krull-Remak-Schmidt that $Soc(X) = 0$. Since $Soc(R_R)$ is essential in $R_R$, this
implies that $X=0$ and so $R$ is right self-injective and hence right PF
by \cite[12.5.2]{Ka}.
\end{proof}




Recall that a module $M$ is said to be finite dimensional if it does not contain an infinite direct sum of nonzero submodules.

\begin{theorem}\label{five} Let $R$ be a ring such that each indecomposable
injective right $R$-module is projective and every projective right
$R$-module is $R$-tight. Then $R$ is a QF ring.
\end{theorem}
\begin{proof} By \cite[Theorem 2.6]{JL}, since each direct sum of indecomposable
injective modules is projective and hence $R$-tight, $R$ has the property
that each finitely generated right $R$-module is finite dimensional. Then,
if $X$ is a finitely generated right $R$-module, $X$ contains an essential
submodule of the form $\oplus_{i=1}^{n} U_i$, where the $U_i$'s are uniform
modules. Thus $E(C) \cong \oplus_{i=1}^{n} E(U_i)$ is a finite direct sum
of indecomposable injective modules and thus, projective. Each $E(C_i)$
has a local endomorphism ring and hence it is a projective local module and,
in particular, cyclic. Thus, we see that $X$ is essentially embeddable in a
finitely generated projective module and by Corollary~\ref{EE}, $R$ is QF.
\end{proof}

\begin{remark} Observe that the hypothesis of the above theorem is weaker than in \cite[Theorem 5.1]{JLOS}, since the assumption that
$R$ is semiperfect (together with the other things in hypothesis) implies that every indecomposable
injective is projective.
\end{remark}






\noindent Our next proposition is a simple yet useful observation.

\begin{proposition}\label{seven} Any left perfect and right tight ring  is right self-injective.
\end{proposition}
\begin{proof} Assume that, on the contrary, $R \neq E(R_R)$, so that there exists
$x \in E(R_R)$ such that $x \notin R$. Then $R_R+xR$ is a finitely generated
submodule of $E(R_R)$ and, as $R$ is right tight, there is an embedding
$R_R+xR \subseteq R_R$. But, since $R_R$ embeds, as a proper submodule,
in $R_R+xR$, we get a proper embedding of $R_R$ into itself, which gives
an infinite descending chain of principal right ideals of $R$, a contradiction.
\end{proof}

\begin{corollary}\label{eight} If $R$ is right artinian and right tight,
then $R$ is QF.
\end{corollary}

\begin{remark}
Note that, however, a right noetherian, right extending and right tight ring does not need to be QF, even in the commutative case, as the example $R = {\bf Z}$ shows.
\end{remark}

\section{An extension of the FGF conjecture.}

\noindent We have studied in the above section different new characterizations of PF rings in terms of ($R$-)tight conditions and we have outlined the strong relation existing between these characterizations and both CF and FGF conjectures. Our purpose in this section is to establish a general problem such that all the existing partial results on CF and FGF conjectures can be included as partial positive answers of this new problem. In order to do so, we are going to define that a ring $R$ is generalized right ($R$-)tight if every finitely generated (resp., cyclic) submodule of $E(R_R)$ embeds in a free module. The results in the above section naturally suggest to propose the following conjecture:

\begin{conjecture} \label{conj1}
If $R$ is a right Kasch right generalized ($R$-)tight ring, then $R_R$ has finitely generated essential socle.
\end{conjecture}

We would like to remark that Conjecture~\ref{conj2} is a consequence of this other conjecture, since if Conjecture~\ref{conj1} is true, then every right cogenerator right $R$-tight ring has finitely generated essential right socle. One can then use the same arguments as in Theorem~\ref{two} to show that $R$ is right PF. On the other hand, Theorem~\ref{one} shows that Conjecture~\ref{conj1} is true if we assume that every cyclic submodule of $E(R_R)$ essentially embeds in a projective module or that $R$ is right extending. Moreover, Osofsky's pioneering characterization of right PF rings \cite[Theorem 1]{O} is a positive solution to the conjecture when $R$ is right self injective. Indeed, it is not difficult to check that most additional conditions which are known to force a right CF ring to be right artinian, also force the above conjecture to be true. We are going to show examples of them.

\begin{proposition}
Let $R$ be a right Kasch right generalized $R$-tight ring. If $R_R$ is noetherian, then $\Soc(R_R)$ is finitely generated and essential in $R$.
\end{proposition}

\begin{proof}
Assume that $R$ is a right generalized right $R$-tight ring. Then every right ideal is a right annihilator. The result now follows from \cite[3.5B Johns' Lemma]{F1} (see also \cite[Theorem~8.9]{NY}).
\end{proof}

\begin{proposition}
Let $R$ be a right Kasch right generalized tight ring. If $R$ is also left Kasch, then $\Soc(R_R)$ is finitely generated and essential in $R$. In particular, any commutative Kasch generalized tight ring has finitely generated essential socle.
\end{proposition}

\begin{proof}
Note that, as $R_R$ is Kasch and generalized tight, any finitely generated right module embeds in a free module. And, as $_RR$ is also Kasch, $R$ is an S-ring in the sense of \cite{Kat}. The result now follows from
\cite[Theorem~1, $(3)\Rightarrow (1)$]{Kat}.
\end{proof}

\noindent One may note that, the assumption that $R_R$ is $R$-tight is critical in the proof of Theorem ~\ref{nil}. We do not know if this theorem is still valid under the weaker assumption that $R_R$ is generalized $R$-tight, and thus give a positive solution to Conjecture~\ref{conj1} when the Jacobson radical of $R/Z(R_R)$ is completely nil. However, our next theorem shows that the arguments can be adapted if we assume the slightly stronger assumption that the Jacobson radical of $R/Z(R_R)$ is completely right T-nilpotent.
Recall that an ideal $I$ of a ring $R$ is called right T-nilpotent if for any infinite sequence $r_1,\ldots,r_n,\ldots$
of elements in $I$, there exists an $n_0\geq 1$ such that
$r_{n_0}\cdot\ldots\cdot r_1=0$
(see \cite[p. 183, Definition before Proposition 2.5]{St}). And we are going to say that the Jacobson radical of a ring $S$ is {\em completely right {\rm T}-nilpotent} if every ring $S'$ which is a homomorphic image of $S$ has right T-nilpotent Jacobson radical.

\begin{theorem}\label{ess-socle} Let $R$ be a right Kasch ring such that any cyclic submodule of $E(R_R)$ embeds in a free module. If  the Jacobson radical of $R/Z(R_R)$ is completely right T-nilpotent, then $\Soc(R_R)$ is finitely generated and essential in $R_R$.
\end{theorem}

\begin{proof}

The general scheme of the proof is quite similar to that of Theorem \ref{nil}. Therefore, we will only elaborate on the part where the proof is different and otherwise we will just refer to the relevant part of the proof of Theorem \ref{nil}. We will follow the same notations as in Theorem \ref{nil}.

\medskip

{\bf Step 1.} We claim that $\Soc(R_R)$ contains only finitely many homogeneous components.


\noindent Let us assume on the contrary that ${\rm Soc}(R_R)$ has infinitely many homogeneous components and let  $\{C_i\}_{i\in I}$ be a representative set of the isomorphism classes of simple modules in ${\rm Soc}(R_R)$.
Proceeding exactly as in the proof of Theorem \ref{nil}, we get $\{e_K+\mathcal{N}\,|\,K\in\mathcal{K}\}$, an orthogonal family of nonzero central idempotents in $S/\mathcal{N}$.

Let $u:R\rightarrow E$ be the inclusion of $R_R$ in its injective envelope and call $x_{K,1}=e_{K}\circ u(1)\in E$. Then $e_{K}\circ u$ factors as

$$
\begin{xy}
  \xymatrix{
    R \ar[d]_u\ar[r]^{f_{K,1}} & x_{K,1}R\ar[d]^{u_{K,1}} \\
    E \ar[r]_{e_K} & E  \\
    }
\end{xy}
$$
where $f_{K,1}$ is an epimorphism and $u_{K,1}$, a monomorphism. As we are assuming that any cyclic  submodule of $E(R_R)$ embeds in a free module (of finite rank), there exists a monomorphism $\alpha_{K,1}:x_{K,1}R\rightarrow R^{(n_{K,1})}$, for some $n_{K,1}\geq 1$, that extends by injectivity to an $s_{K,1}:E\rightarrow E^{(n_{K,1})}$ such that $u^{(n_{K,1})}\circ\alpha_{K,1}=s_{K,1}\circ u_{K,1}$, where we are denoting by $u^{(n_{K,1})}:R^{(n_{K,1})}\rightarrow E^{(n_{K,1})}$ the inclusion. Again, as $s_{K,1}|_{E(x_{K,1}R)}:E(x_{K,1}R)\rightarrow E$ is injective, there exists an $h_{K,1}:E^{(n_{K,1})}\rightarrow E$ such that $h_{K,1}\circ s_{K,1}\circ e_{K}=e_{K}$

$$\begin{xy}
  \xymatrix{
   R \ar[d]_u\ar[r]^{f_{K,1}} & x_{K,1}R\ar[d]^{u_{K,1}}\ar [r]^{\alpha _{K,1}} &
   R^{(n_{K,1})}\ar[d]^{u^{(n_{K,1})}} \\
    E \ar[r]_{e_K} & E\ar[r]_{s_{K,1}}& E^{(n_{K,1})}  \ar[r]_{h_{K,1}} &E.\\
    }
\end{xy}
$$

Let $\pi_{K,1,t}:E^{(n_{K,1})}\rightarrow E$ and $v_{K,1,t}:E\rightarrow E^{(n_{K,1})}$, for $t=1,\ldots,n_{K,1}$ be the canonical projections and injections.  Then $\sum_{t=1}^{n_{K,1}}v_{K,1,t}\circ \pi_{K,1,t}=1_{E^{(n_{K,1})}}$. Therefore, we have that
$$\begin{array}{ll}
e_K &=h_{K,1}\circ s_{K,1}\circ e_K   \\
 & = h_{K,1}\circ \left(\sum_{t=1}^{n_{K,1}}v_{K,1,t}\circ \pi_{K,1,t}\right)\circ s_{K,1}\circ e_K  \\
& = \sum_{t=1}^{n_{K,1}}\left(h_{K,1,}\circ v_{K,1,t}\circ \pi_{K,1,t}\circ s_{K,1}\circ e_K\right)
\end{array}
$$

As $e_K+J\notin \mathcal{N}/J$, there exists a $t_0$ such that
$$\left(h_{K,1}\circ v_{K,1,{t_0}}\circ \pi_{K,1,{t_0}}\circ s_{K,1}\circ e_K\right)+J\notin \mathcal{N}/J.$$

Call $s'_{K,1}=\pi_{K,1,t_0}\circ s_ {K,1}$, $h'_{K,1}=h_{K,1}\circ v_{K,1,t_0}$ and $r_{K,1}= s'_{K,1}\circ e_K\circ u(1)\in R$.
And let $x_{K,2}=h'_{K,1}\circ u(r_{K,1})= h'_{K,1}\circ s'_{K,1}\circ e_K\circ u(1)\in E$. Again, as any cyclic submodule of $E(R_R)$ embeds in a free module, there exists a monomorphism $\alpha_{K,2}:x_{K,2}R\rightarrow R^{(n_{K,2})}$, for some $n_{K,2}\geq 1$, that extends by injectivity to an $s_{K,2}:E\rightarrow E^{(n_{K,2})}$ such that $u^{(n_{K,2})}\circ\alpha_{K,2}=s_{K,2}\circ u_{K,2}$. And, as $s_{K,2}|_{E(x_{K,2}R)}:E(x_{K,2}R)\rightarrow E^{(n_{K,2})}$ is injective, there exists an $h_{K,2}:E^{(n_{K,2})}\rightarrow E$ such that $h_{K,2}\circ s_{K,2}|_{E(x_{K,2}R)}=1|_{E(x_{K,2}R)}$. Moreover, as by construction,
$$h_{K,2}\circ s_{K,2}\circ h'_{K,1}\circ s'_{K,1}\circ e_K\circ u = h'_{K,1}\circ s'_{K,1}\circ e_K\circ u,$$
we get that $h_{K,2}\circ s_{K,2}\circ h'_{K,1}\circ s'_{K,1}\circ e_K - h'_{K,1}\circ s'_{K,1}\circ e_K\in J(S)$. Therefore, $h_{K,2}\circ s_{K,2}\circ h'_{K,1}\circ s'_{K,1}\circ e_K+J\notin \mathcal{N}/J$, as neither is $h'_{K,1}\circ s'_{K,1}\circ e_K + J\notin \mathcal{N}/J$.

Let $\pi_{K,2,t}:E^{(n_{K,2})}\rightarrow E$ and $v_{K,2,t}:E\rightarrow E^{(n_{K,2})}$, for $t=1,\ldots,n_{K,2}$ be the canonical projections and injections.  As
before, there exists a $t_0$ such that such that
$$\left(h_{K,2}\circ v_{K,2,{t_0}}\circ \pi_{K,2,{t_0}}\circ s_{K,2}\circ h'_{K,1}\circ s'_{K,1}\circ e_K\right)+J\notin \mathcal{N}/J.$$

Call now $s'_{K,2}=\pi_{K,2,t_0}\circ s_ {K,2}$, $h'_{K,2}=h_{K,2}\circ v_{K,2,t_0}$ and $r_{K,2}= s'_{K,2}\circ h_{K,1}\circ u(1)\in R$.

Repeating the same arguments, we can define $s'_{K,l},h'_{K,l}\in S$ and $r_{K,l}\in R$, for each $l\geq 1$ such that:
\begin{itemize}
\item $(h'_{K,l}\circ s'_{K,l}\circ\ldots\circ h'_{K,1}\circ s'_{K,1}\circ e_K)+J\notin \mathcal{N}/J$ for any $l\geq 1$.
\item $\Psi(r_{K,1}+\mathcal{M})=s'_{K,1}\circ e_K+\mathcal{N}$.
\item $\Psi(r_{K,l+1}+\mathcal{M})=s'_{K,l+1}\circ h'_{K,l}+\mathcal{N}$ for each $l\geq 1$.
\end{itemize}
We claim that there exists some $l\geq 1$ such that $r_{K,l}+\mathcal{M}\notin J(R/\mathcal{M})$. Assume on the contrary that $r_{K,l}+\mathcal{M}\in J(R/\mathcal{M})$ for each $l\geq 1$. As we are assuming that the Jacobson radical of $R/Z(R_R)$ is completely right T-nilpotent, there exists an $l_0$ such that $r_{K,l_0+1}\cdot\ldots\cdot r_{K,1}+\mathcal{M}=0$ in $R/\mathcal{M}$. But this means that
$$(s'_{K,l_0+1}\circ h'_{K,l_0}\circ\ldots\circ s'_{K,1}\circ e_{K})+\mathcal{N}=\Psi\left((r_{K,l_0+1}\cdot\ldots\cdot r_{K,1}+\mathcal{M})\right)=0$$
in $S/\mathcal{N}$, a contradiction. This proves our claim.

Call $r_K=r_{K,l_0}$, $s'_K=s'_{K,l_0+1}$ and $h'_K=h'_{K,l_0}$. As $e_K+J$ is central in $S/J$, we deduce that $s'_K\circ h'_K \circ e_K+J\notin S/\mathcal{N}$.
On the other hand, as $r_K+\mathcal{M}\notin J(R/\mathcal{M})$, there exists a maximal right ideal $L_K/\mathcal{M}$ of $R/\mathcal{M}$ such that $r_K+\mathcal{M}\notin L_K/\mathcal{M}$ and thus, $R/L_K$, is a simple right $R/\mathcal M$-module such that $R/L_K\cdot (r_K+\mathcal{M})\neq 0$.

Now as in the proof of Theorem \ref{nil}, we show that $e_K\circ e_{K'}\notin \mathcal{N}$. But this is a contradiction, since $e_K\circ e_{K'}\in \mathcal{N}$ by construction.

We have constructed then $|K|$ isomorphism classes of simple right $R$-modules. This yields a contradiction, since $|\mathcal{K}|\gvertneqq |I|$.

\medskip

{\bf Step 2.} We claim that any homogeneous component of ${\rm Soc}(R_R)$ is finitely generated.


\noindent Assume on the contrary that there exists a homogeneous component which is not finitely generated. We already know that there are only finitely many homogeneous components in ${\rm Soc}(R_R)$. Let $\{C_1,\ldots,C_m\}$ be a representative set of simple modules belonging to them.

Let us fix a simple module $C_i$, for $i=1,\ldots,m$, and call $C_{i,1}=C_i$. Let $e_{i,1}\in S$ be an idempotent such that $E(C_{i,1})=e_{i,1}E$ and let us repeat the reasonings made in Step~1, in order to decompose $e_{i,1}\circ u=u_{i,1}\circ f_{i,1}$, where $f_{i,1}:R\rightarrow x_{i,1}R$ is an epimorphism and $u_{i,1}:x_{i,1}R\rightarrow E$, a monomorphism. As each cyclic submodule of $E(R_R)$ embeds in a free module, there exists a monomorphism $\alpha:x_{i,1}R\rightarrow R^n$, for some $n\geq 1$. So there exists a projection $\pi:R^n\rightarrow R$ such that $\pi\circ \alpha|_{C_{i,1}}\neq 0$.
And, as the simple module $C_{i,1}$ is essential in $x_{i,1}R$, we deduce that $\pi\circ\alpha$ is a monomorphism. Call it $\alpha_{i,1}$. This monomorphism extends by injectivity to an endomorphism $s_{i,1}:E\rightarrow E$ such that $s_{i,1}\circ u_{i,1}=u\circ \alpha_{i,1}$. Note that $s_{i,1}\circ e_{i,1}\notin J(S)$ since $s_{i,1}\circ e_{i,1}|_{C_{i,1}}$ is a monomorphism and thus, $s_{i,1}\circ e_{i,1}$ does not have essential kernel.

Let us call $C_{i,2}=s_{i,1}\circ e_{i,1}(C_{i,1})$. Then $C_{i,2}$ is a simple right $R$-module isomorphic to $C_{i,1}$ and therefore, we can repeat the above construction to obtain $e_{i,2},s_{i,2}\in S$ such that $(s_{i,2}\circ e_{i,2})|_{C_{i,2}}$ is injective and therefore, $(s_{i,2}\circ e_{i,2}\circ s_{i,1}\circ e_{i,1})|_{C_{i,1}}$ is also injective. So $s_{i,2}\circ e_{i,2}\circ s_{i,1}\circ e_{i,1}\notin J(S)$. Repeating this construction, we can find by recurrence idempotents $e_{i,n}\in S$ and elements $s_{i,n}\in S$ such that
$(s_{i,n}\circ e_{i,n}\circ\ldots\circ s_{i,1}\circ e_{i,1})|_{C_i,1}$ is a monomorphism, and therefore, $s_{i,n}e_{i,n}\circ\ldots\circ s_{i,1}\circ e_{i,1}\notin J(S)$.

Let $r_{i,n}=s_{i,n}\circ e_{i,n}\circ u(1)\in R$ for each $n\geq 1$. We claim that $r_{i,n}+Z(R_R)\notin J(R/Z(R_R))$ for some $n\geq 1$.
Assume on the contrary that $r_{i,n}+Z(R_R)\in J(R/Z(R_R))$ for every $n\geq 1$.
As we are assuming that $J(R/Z(R_R))$ is completely right T-nilpotent, there exists an $n_0$ such that $r_{i,n_0}\cdot\ldots\cdot r_{i,1}+Z(R_R)=0$ in $R/Z(R_R)$. Thus, $\Phi(r_{i,n_0}\ldots r_{i,1}+Z(R_R))=0$ in $S/J$. But, as by construction, $\Phi(r_{i,n}+Z(R_R))=s_{i,n}\circ e_{i,n}+J$, we deduce that $(s_{i,n_0}\circ e_{i,n_0}\circ\ldots\circ s_{i,1}\circ e_{i,1}) \in J(S)$, a contradiction. This proves our claim.

Let us choose an $n$ such that $r_{i,n}+Z(R_R)\notin J(R/Z(R_R))$ and set $r_i=r_{i,n},\,e_i=e_{i,n},\, s_i=s_{i,n}$. Replacing, if necessary $C_i$ by its isomorphic image $C_{i,n}$, we get that $E(C_i)=e_iE$. Moreover, as $r_i+Z(R_R)\notin J(R/Z(R_R))$, there exists a maximal right ideal $L_i/Z(R_R)$ of $R/Z(R_R)$ such that $r_i+Z(R_R)\notin L_i/Z(R_R)$. Call $D_i=R/L_i$. Then $D_i$ is a simple $R/Z(R_R)$-module such that $D_i\cdot (r_i+Z(R_R))\neq 0$.

We claim that $D_i\ncong D_j$ if $i\neq j$. Assume on the contrary that $\delta:D_i\rightarrow D_j$ is an isomorphism. As $D_i\cdot(r_i+Z(R_R))\neq 0$, we can choose an $0\neq x\in D_i$ such that  $x\cdot(r_i+Z(R_R))\neq 0$ and thus, $\delta(x)\cdot (r_i+Z(R_R))$ is a generator of $D_j$. Again, as $D_j\cdot (r_j+Z(R_R))\neq 0$, there exists an $r\in R$ such that $\delta(x)(r_irr_j+Z(R_R))\neq 0$. In particular, $r_irr_j\notin Z(R_R)$. And thus,
$$(s_i\circ e_i+J)\circ \Phi(r+Z(R_R))\circ (s_j\circ e_j+J)=\Phi(r_irr_j+ Z(R_R))\neq 0.$$
Let
$$g=(s_i+J)\circ (e_i+J)\circ \Phi(r+Z(R_R))\circ (s_j+J):S/J\longrightarrow S/J.$$ Then, as ${\rm Im}((e_i+J)\circ \Phi(r+Z(R_R))\circ (s_j+J))$ is contained in the simple right ideal $(e_iS+J)/J$, we deduce that ${\rm Im}(g)$ is contained in a simple right ideal of $S/J$, say $Y$, isomorphic to $(e_iS+J)/J$. But then, as
$(s_i\circ e_i+J)\circ \Phi(r+Z(R_R))\circ (s_j\circ e_j+J)\neq 0$, we deduce that the homomorphism $(s_j+J)\circ(e_j+J)\circ \Phi(r+Z(R_R))\circ (s_i+J)$ is not zero either, when restricted to the simple right ideal $(e_jS+J)/J$. And this means that $(e_jS+J)/J\cong Y\cong (e_iS+J)/J$. Therefore, $e_iS\cong e_jS$, since they are the projective covers of $(e_iS+J)/J$ and $(e_jS+J)/J$, respectively. But this means that $E(C_i)\cong E(C_j)$ and so, $C_i\cong C_j$, a contradiction. This proves our claim.

Lets now assume that some homogeneous component is not finitely generated. Say that it is the homogeneous component associated to $C_1$. Then as in the proof of Theorem \ref{nil}, we arrive at a contradiction. This shows that each homogeneous component of ${Soc}(R_R)$ is finitely generated and thus, ${Soc}(R_R)$ is finitely generated.

\medskip

{\bf Step 3.} We finally claim that $\Soc (R_R)$ is essential in $R_R$.

\noindent The proof of this part is identical to the step 3 of Theorem \ref{nil} and this completes the proof.

\end{proof}

We can now state the following partial answer to the CF and FGF conjectures.

\begin{corollary}\label{main} Let $R$ be a right CF ring. If the Jacobson radical of $R/Z(R_R)$ is completely right T-nilpotent, then $R$ is right artinian.
In particular, if R is right FGF, then it is QF.
\end{corollary}

\begin{proof}
By Theorem~\ref{ess-socle}, we know that $R$ has a finitely generated essential right socle and thus, any cyclic right $R$-module also has a finitely generated essential socle, since it embeds in a (finitely generated) free module. Therefore, $R$ is right artinian (see e.g. \cite[Theorem 10.4]{AF} and \cite[Proposition 10.10]{AF}).
\end{proof}

Our next corollary shows that the main result of \cite{GT} is a consequence of our Corollary~\ref{main}.

\begin{corollary}
Let $R$ be a ring and assume that $R/Z(R_R)$ is a von Neumann regular ring. Then:
\begin{enumerate}
\item If $R$ is right CF, then it is right artinian.

\item If $R$ is right FGF, then it is QF.

\end{enumerate}
\end{corollary}
\begin{proof}
If $R/Z(R_R)$ is von Neumann regular, then every ring which is a homomorphic image of $R/Z(R_R)$ has zero Jacobson radical. We may now apply Corollary \ref{main}.
\end{proof}

We close the paper by extending \cite[Theorem 2]{SR}.

\begin{corollary}
Let $R$ be a right CF ring such that $R/J(R)$ is von Neumann regular and $J(R)$ is right T-nilpotent. Then $R$ is right artinian. In particular, if $R$ is right FGF, then it is QF.
\end{corollary}

\begin{proof}
As $J(R)$ is right T-nilpotent, idempotents lift modulo $J(R)$ (see e.g. \cite[Proposition 4.2]{St}). So the ring is semiregular.
We claim that $Z(R_R)\subseteq J(R)$. Choose any $r\in Z(R_R)$ and call $f:R\rightarrow R$ the homomorphism given by left multiplication by $r$.
As $rR$ is finitely generated, there exists an idempotent $e\in R$ such that $eR\subseteq xR$ and $xR\cap (1-e)R$ is superfluous in $R$ (see \cite[Theorem 1.6]{N}). But, as ${\rm Ker}(f)$ is essential in $R$, $e=0$, and this means that $xR=xR\cap (1-e)R$ is superfluous in $R$. So $r\in J(R)$.

Now, let $M/Z(R_R)$ be any two-sided ideal of $R/Z(R_R)$. As $Z(R_R)\subseteq J(R)$, we have a surjective homomorphism of rings $\varphi: R/M \rightarrow R/(M+J)$. And then, $\varphi(J(R/M))\subseteq J(R/(J+M))$ (see \cite[Corollary 15.8]{AF}). Therefore, any element $r+M\in J(R/M)$ is of the form $j+M$ with $j\in J(R)$. Thus, $J(R/M)$ is right T-nilpotent. The result now follows from Corollary~\ref{main}.
\end{proof}

We close the paper by extending two results of \cite{GT}.

\begin{corollary}(see \cite[Theorem 16]{GT})\label{CF-characterization}
Let $R$ be a right CF ring. Then the following conditions are equivalent:
\begin{enumerate}
\item $J(R/Z(R_R))$ is completely right T-nilpotent.
\item Every cyclic right $R$-module essentially embeds in a projective module.
\item $R_R$ is continuous.
\end{enumerate}
Moreover, in this case $Z(R_R)=J(R)$ and $R$ is right artinian.
\end{corollary}
\begin{proof}
$(1)\Rightarrow (2)$ is a consequence of Theorem~\ref{main} and \cite[Corollary 13]{GT}. The other implications follow from \cite[Theorem 16]{GT}.
\end{proof}

\begin{corollary}
Let $R$ be a ring. Then the following conditions are equivalent:
\begin{enumerate}
\item Every cyclic right $R$-module embeds in $R$ and $J(R/Z(R_R))$ is completely  right nil.

\item Every cyclic right $R$-module essentially embeds in a direct summand of $R$.

\item $R$ is a direct sum of rings which are either right uniserial or finite matrix rings over two-sided uniserial rings.

\end{enumerate}
\end{corollary}
\begin{proof}
$(1)\Rightarrow (2)$ is a consequence of Theorem~\ref{nil} and the proof of \cite[Theorem 17]{GT}. The other implications follow from  \cite[Theorem 17]{GT}.
\end{proof}

\bigskip

\noindent {\bf Acknowledgment.}

\medskip

\noindent Part of this paper was written during a visit of the first author to the Gebze Institute of Technology in Turkey during August 2014 and a visit of the second author to the Department of Mathematics of the University of Murcia in October and November 2014. They would like to thank both departments for their hospitality during their visits.


\bigskip

\bigskip


\end{document}